\documentclass[12pt]{amsart}

\usepackage{color}
\usepackage{amsmath, amsthm, amssymb}
\usepackage{amsfonts}
\usepackage{accents}
\usepackage[ansinew]{inputenc}
\usepackage[dvips]{epsfig}
\usepackage{graphicx}
\usepackage[english]{babel}
\usepackage{hyperref}
\theoremstyle{plain}
\newtheorem{thm}{Theorem}[section]
\newtheorem{cor}[thm]{Corollary}
\newtheorem{lem}[thm]{Lemma}
\newtheorem{prop}[thm]{Proposition}
\newtheorem{claim}[thm]{Claim}

\theoremstyle{definition}
\newtheorem{defi}[thm]{Definition}

\theoremstyle{remark}
\newtheorem{rem}[thm]{Remark}

\numberwithin{equation}{section}

\newcommand{\average}{{\mathchoice {\kern1ex\vcenter{\hrule height.4pt
width 6pt depth0pt} \kern-9.7pt} {\kern1ex\vcenter{\hrule
height.4pt width 4.3pt depth0pt} \kern-7pt} {} {} }}

\def\R{\mathbb{R}}

\begin{document}

\title[Pohozaev identities for higher order fractional Laplacians]{Local integration by parts and Pohozaev identities for higher order fractional Laplacians}

\author{Xavier Ros-Oton}

\address{Universitat Polit\`ecnica de Catalunya, Departament de Matem\`{a}tica  Aplicada I, Diagonal 647, 08028 Barcelona, Spain}
\email{xavier.ros.oton@upc.edu}

\thanks{The authors were supported by grants MTM2008-06349-C03-01, MTM2011-27739-C04-01 (Spain), and 2009SGR345 (Catalunya)}

\author{Joaquim Serra}

\address{Universitat Polit\`ecnica de Catalunya, Departament de Matem\`{a}tica  Aplicada I, Diagonal 647, 08028 Barcelona, Spain}

\email{joaquim.serra@upc.edu}

\keywords{Fractional Laplacian, Pohozaev identity, integration by parts.}

\maketitle

\begin{abstract}
We establish an integration by parts formula in bounded domains for the higher order fractional Laplacian $(-\Delta)^s$ with $s>1$.
We also obtain the Pohozaev identity for this operator.
Both identities involve local boundary terms, and they extend the identities obtained by the authors in the case $s\in(0,1)$.

As an immediate consequence of these results, we obtain a unique continuation property for the eigenfunctions $(-\Delta)^s\phi=\lambda\phi$ in $\Omega$, $\phi\equiv0$ in $\R^n\setminus\Omega$.
\end{abstract}

\section{Introduction and results}

We consider bounded solutions $u\in H^s(\R^n)$ to the Dirichlet problem
\begin{equation}\label{eq}
\left\{ \begin{array}{rcll}
(-\Delta)^s u &=&f(x,u)&\textrm{in }\ \Omega \\
u&=&0&\textrm{in }\ \R^n\backslash \Omega\end{array}\right.
\end{equation}
for the higher order fractional Laplacian $(-\Delta)^s$, $s>1$.
Here, and in the rest of the paper, $\Omega\subset\R^n$ is any bounded smooth domain, $f$ is continuous, and $(-\Delta)^s$ is defined by
\[\widehat{(-\Delta)^s u}(\xi)=|\xi|^{2s}\widehat{u}(\xi)\quad \textrm{for a.e.}\ \xi\in\R^n.\]
Equivalently, it can be defined inductively by $(-\Delta)^s=(-\Delta)^{s-1}\circ(-\Delta)$, once $(-\Delta)^s$ is defined for $s\in(0,1)$ ---see for example \cite{S,RS} for its definition for $s\in(0,1)$ in terms of an integral formula.

This higher order operator appears in PDEs \cite{KP,Tao,Y,Grubb,MYZ,Sjolin,EIK}, Geometry \cite{Martinazzi,ChangYang,GrahamZ}, Analysis \cite{CT,Samko}, and also in applied sciences \cite{MABZ,ZH}.

The aim of this paper is to establish an integration by parts formula in bounded domains $\Omega$ for the operator $(-\Delta)^s$ with $s>1$, as well as the Pohozaev identity for problem \eqref{eq}.
The results of the present paper extend the identities obtained by the authors in \cite{RS-Poh} for $s\in(0,1)$ to higher order fractional Laplacians.
They also extend the Pohozaev identities established by Pucci and Serrin \cite{PSerrin} for problem \eqref{eq} when $s=k\geq2$ is an integer.

Identities of Pohozhaev type have been widely used in the analysis of PDEs \cite{R,P,V,DS,M,PSerrin}.
These identities are used to show sharp nonexistence results, monotonicity formulas, energy estimates for ground states in $\R^n$, unique continuation properties, radial symmetry of solutions, or uniqueness results.
Moreover, they are frequently used in control theory \cite{Caz}, wave equations \cite{BO}, geometry \cite{Sc,KW,Pol}, and harmonic maps \cite{CZ}.
For example, the controllability of the linear wave equation $u_{tt}-\Delta u=0$ uses energy estimates (in terms of the boundary contribution) which follow from the Pohozaev identity.
Similar identities are used in \cite{Strauss} to the study of the decay, stability, and scattering of waves
in nonlinear media.
Also, in some critical geometric equations like $-\Delta u=Ke^{2u}$ in $\R^2$ or $-\Delta u=u^{\frac{n+2}{n-2}}$ in $\R^n$, $n\geq3$, Pohozaev identities are essential to show concentration-compactness phenomena.

In our paper \cite{RS-Poh}, we established the Pohozaev identity for the fractional Laplacian $(-\Delta)^s$ with $s\in(0,1)$.
Integro-differential equations involving this operator arise when studying stochastic processes with jumps, and they are used to model anomalous diffusions, prices of assets in financial mathematics, and other physical phenomena with long range interactions.

After the publication of our results \cite{RS-CRAS,RS-Poh}, we have been asked if these identities hold also for $s>1$.
This is because in some contexts it is natural to study the higher order operator $(-\Delta)^s$ with $s>1$.
For example, in the fractional wave equation $u_{tt}+(-\Delta)^su=0$ it is necessary to have $s\geq1$ in order to guarantee a uniform speed of propagation.
Moreover, in Geometry it is of special importance the case $s=n/2$, because it appears in the prescribed $Q$-curvature equation $(-\Delta)^{n/2}u=Ke^{nu}$.

Before stating our identities, we recall the important recent regularity results of G. Grubb \cite{Grubb}, which are essential in the proofs (and even in the statement) of our identities.
We also have to introduce the following function $d(x)$ ---it will be like $\textrm{dist}(x,\R^n\setminus\Omega)$ but smoothing out its possible singularities inside $\Omega$.

\begin{defi}\label{d}
Given a bounded smooth domain $\Omega\subset\R^n$, $d(x)$ will be any function that is positive in $\Omega$ and $C^\infty(\overline\Omega)$, that coincides with $\textrm{dist}(x,\R^n\setminus\Omega)$ in a neighborhood of $\partial\Omega$, and that $d\equiv0$ in $\R^n\setminus\Omega$.
\end{defi}

\begin{thm}[\cite{Grubb}]\label{thGrubb}
Let $\Omega\subset\R^n$ be a bounded smooth domain, $d(x)$ be as in Definition \ref{d}, and $s>0$.
Let $g\in L^\infty(\Omega)$, and let $u\in H^s(\R^n)$ be the solution to the homogeneous Dirichlet problem
\begin{equation}\label{eqlin}
\left\{ \begin{array}{rcll}
(-\Delta)^s u &=&g(x)&\textrm{in }\Omega \\
u&=&0&\textrm{in }\mathbb R^n\backslash \Omega.
\end{array}\right.
\end{equation}
Then, for all $\epsilon\in(0,1)$ and all $\alpha\in(0,+\infty)$ there exist a constant $C$ such that
\[\|u/d^s\|_{C^{s-\epsilon}(\overline\Omega)}\leq C\|g\|_{L^\infty(\Omega)}\]
and
\[\|u/d^s\|_{C^{\alpha+s-\epsilon}(\overline\Omega)}\leq C\|g\|_{C^\alpha(\overline\Omega)}.\]
Moreover,
\[g\in C^\infty(\overline\Omega)\ \Longleftrightarrow \ u/d^s\in C^\infty(\overline\Omega).\]
In particular, $(-\Delta)^sd^s$ is smooth in $\overline\Omega$.
\end{thm}

The results of Grubb apply also to nonhomogeneous Dirichlet problems and to more general pseudodifferential operators and functional spaces; see \cite{Grubb}.

\begin{rem}\label{rem1}
Notice that we are considering weak solutions $u\in H^s(\R^n)$ to
\begin{equation}\label{eqlinear}
\left\{ \begin{array}{rcll}
(-\Delta)^s u &=&g(x)&\textrm{in }\ \Omega \\
u&=&0&\textrm{in }\ \R^n\backslash \Omega.
\end{array}\right.\end{equation}
The unique solution $u\in H^s(\R^n)$ to this problem can be obtained by minimizing the energy functional
\[\mathcal E(u)=\frac12\|u\|_{\accentset{\circ}{H}^s(\R^n)}^2-\int_\Omega gu=\frac12\int_{\R^n}|(-\Delta)^{s/2}u|^2-\int_\Omega gu\]
among all functions $u\in H^s(\R^n)$ satisfying $u\equiv0$ in $\R^n\setminus\Omega$.

It is important to notice that the condition $u\in H^s(\R^n)$ is necessary in order to have uniqueness of solutions ---see Remark \ref{rem-bis} at the end of the Introduction.
\end{rem}

Our first result is the following Pohozaev type identity.

\begin{thm}\label{intparts}
Let $\Omega$ be a bounded smooth domain, $d(x)$ be as in Definition \ref{d}, and $s>1$.
Let $u\in H^s(\R^n)$ be such that $u\equiv0$ in $\R^n\setminus\Omega$ and that $(-\Delta)^s u$ belongs to $L^\infty(\Omega)$.

Then, $u/d^s|_{\Omega}$ has a continuous extension to $\overline\Omega$, and the following identity holds
\begin{equation}\label{pohozaev-id}
\int_\Omega(x\cdot\nabla u)(-\Delta)^su\ dx=\frac{2s-n}{2}\int_{\Omega}u(-\Delta)^su\ dx-\frac{\Gamma(1+s)^2}{2}\int_{\partial\Omega}\left(\frac{u}{d^s}\right)^2(x\cdot\nu)d\sigma.
\end{equation}
Here, $\nu$ is the unit outward normal to $\partial\Omega$ at $x$ and $\Gamma$ is the Gamma function.
\end{thm}

Before our work \cite{RS-Poh} for $s\in(0,1)$, no integration by parts type formula for the fractional Laplacian involving a local boundary term as in \eqref{pohozaev-id} was known.
In particular, there was not even a candidate for a Pohozaev identity in bounded domains.
We found and succeeded to establish this identity for $s\in(0,1)$ in \cite{RS-Poh} by introducing a new method.
Here, we adapt it to the case $s>1$.
To our knowledge, Theorem \ref{intparts} is new even in dimension $n=1$.

For the Laplacian $-\Delta$, the Pohozaev identity follows easily from integration by parts or the divergence theorem.
However, in this nonlocal framework these tools are not available, and our proof is more involved.
In this direction, the constant $\Gamma(1+s)^2$ in \eqref{pohozaev-id} seems to indicate that, in contrast with the case $s=1$, it can not follow immediately from basic vector calculus identities.

As a consequence of Theorem \ref{intparts}, we also obtain a new integration by parts formula in bounded domains.
It reads as follows.

\begin{thm}\label{corintparts}
Let $\Omega$ be a bounded smooth domain, $d(x)$ be as in Definition \ref{d}, and $s>1$.
Let $u$ and $v$ be $H^s(\R^n)$ functions satisfying that $u\equiv v\equiv0$ in $\R^n\setminus\Omega$ and that $(-\Delta)^s u$ and $(-\Delta)^sv$ belong to $L^\infty(\Omega)$.

Then, $u/d^s|_{\Omega}$ and $v/d^s|_{\Omega}$ have continuous extensions to $\overline\Omega$, and the following identity holds
\[\int_\Omega (-\Delta)^su\ v_{x_i}\,dx=-\int_\Omega u_{x_i}(-\Delta)^sv\,dx-\Gamma(1+s)^2\int_{\partial\Omega}\frac{u}{d^{s}}\frac{v}{d^{s}}\,\nu_i\,d\sigma\]
for $i=1,...,n$, where $\nu$ is the unit outward normal to $\partial\Omega$ at $x$, and $\Gamma$ is the Gamma function.
\end{thm}

(In Theorem \ref{corintparts}, we have corrected the sign on the boundary contribution, which was incorrectly stated in the previous version of the paper.)

\begin{rem}\label{remGrubb}
In personal communication, G. Grubb explained to us how the Fourier methods in her paper \cite{Grubb}, at least in the case of a flat piece of the boundary, can be used to show Theorem \ref{corintparts}.
It would be very interesting to understand better this relation, and to see what kind of Pohozaev-type identities can be obtained via Fourier transform methods.
\end{rem}

Applying Theorem \ref{intparts} to solutions of semilinear problems of the form \eqref{eq} we find the Pohozaev identity for the higher order fractional Laplacian.
We state next this result which, for the sake of simplicity, we state for solutions to \eqref{eq2} below instead of \eqref{eq}.

When $s=k\geq2$ is an integer, the following identity was obtained by Pucci and Serrin \cite{PSerrin}.

\begin{cor}\label{thpoh}
Let $\Omega\subset\R^n$ be any bounded smooth domain, $f$ be a continuous nonlinearity, and $s>1$.
Let $u\in H^s(\R^n)$ be any bounded solution of
\begin{equation}\label{eq2}
\left\{ \begin{array}{rcll}
(-\Delta)^s u &=&f(u)&\textrm{in }\Omega \\
u&=&0&\textrm{in }\mathbb R^n\backslash \Omega,\end{array}\right.
\end{equation}
and let $d$ be as in Definition \ref{d}.
Then, $u/d^s|_{\Omega}$ has a continuous extension to $\overline\Omega$, and the following identity holds
\[(2s-n)\int_{\Omega}uf(u)dx+2n\int_\Omega F(u)dx=\Gamma(1+s)^2\int_{\partial\Omega}\left(\frac{u}{d^s}\right)^2(x\cdot\nu)d\sigma,\]
where $F(t)=\int_0^tf$, $\nu$ is the unit outward normal to $\partial\Omega$ at $x$, and $\Gamma$ is the Gamma function.
\end{cor}

In case that $f(u)$ is supercritical (e.g. $f(u)=|u|^{p-1}u$ with $n>2s$ and $p>\frac{n+2s}{n-2s}$), this identity yields the nonexistence of bounded solutions $u\in H^s(\R^n)$ to problem \eqref{eq2}; see \cite{RS-Poh,RS-nonexistence}.
When $s\leq 1$, the nonexistence of bounded \emph{positive} solutions to the critical problem with $f(u)=u^{\frac{n+2s}{n-2s}}$ can be also deduced from this identity by using the Hopf Lemma ---which yields that $u/d^s>0$ on $\partial\Omega$.
However, when $s>1$ the maximum principle does not hold, and thus no Hopf Lemma holds for \eqref{eq2}.
Therefore, the nonexistence of solutions for the critical problem in case $s>1$ is related to a general unique continuation property of the form $u/d^s\equiv0$ on $\partial\Omega$ $\Longrightarrow$ $u\equiv0$ in $\Omega$.
This is not known even for the case $s\in(0,1)$.

In case that $f(u)$ is subcritical, it turns out that this unique continuation property follows from Corollary \ref{thpoh}.
Because of its special importance, we state this result for the case of the eigenfunctions, i.e., for $f(u)=\lambda u$.

\begin{cor}\label{cor-uniquecont}
Let $\Omega\subset \R^n$ be a bounded smooth domain.
Let $s>1$, and let $\phi\in H^s(\R^n)$ be any solution of
\[\left\{ \begin{array}{rcll}
(-\Delta)^s \phi &=&\lambda\phi&\textrm{in }\Omega \\
\phi&=&0&\textrm{in }\mathbb R^n\backslash \Omega.\end{array}\right.\]
Then,
\[\left.\frac{\phi}{d^s}\right|_{\partial\Omega}\equiv0\quad \textrm{on}\quad \partial\Omega\qquad \Longrightarrow\qquad \phi\equiv0\quad \textrm{in}\quad \Omega.\]
Here, $\phi/d^s$ on $\partial\Omega$ has to be understood as a limit (as in Theorem \ref{intparts}).
\end{cor}

Notice that this Corollary is stated here for $s>1$, but the same result holds also for $s\in(0,1)$ ---in that case one must use our results in \cite{RS-Poh} instead of Corollary \ref{thpoh}.

For elliptic operators of second order, this type of unique continuation property is an essential ingredient in the proof of the generic simplicity of the spectrum \cite{U} which, in turn, arises when analyzing controllability and stabilization issues for evolution problems \cite{L}.

To prove the generic simplicity of the full spectrum one needs to compute a shape derivative.
It is in this computation where the boundary term $u/d^s|_{\partial\Omega}$ arises naturally.
See for example \cite{OZ}, where this was done for the bilaplacian $\Delta^2$ ---which corresponds to $s=2$ in our result.
For the fractional Laplacian $(-\Delta)^s$, this shape derivative has been computed for $s=1/2$ and $n=2$ in \cite{DG}.

\begin{rem}\label{rem-bis}
As said in Remark \ref{rem1}, problem \eqref{eqlinear} admits a unique $H^s(\R^n)$ solution, but when $s>1$ there is no uniqueness of bounded solutions for the problem.
Indeed, for all $s>0$ the following nonhomogeneous Dirichlet problem \cite{Grubb,A} is well posed (in the appropriate energy space):
\begin{equation}\label{eqlinear2}
\left\{ \begin{array}{rcll}
(-\Delta)^s u &=&g(x)&\textrm{in }\ \Omega \\
u&=&0&\textrm{in }\ \R^n\backslash \Omega \\
u/d^{s-1} &=&\varphi(x)&\textrm{on }\ \partial\Omega.
\end{array}\right.\end{equation}
The well-posedness of this problem was proved at nearly the same time and with completely independent methods by Grubb \cite{Grubb} (for all $s>0$) and by Abatangelo \cite{A} \footnote{The problem studied in \cite{A} was not exactly formulated as \eqref{eqlinear2}, but it can be shown that in fact the two problems are equivalent.} (for $s\in(0,1)$).

Notice that the quantity $u/d^{s-1}$ on $\partial\Omega$ has to be understood as a limit.
In particular, solutions to these nonhomogeneous problems behave in general like $d^{s-1}$ near the boundary, while the unique solution $u\in H^s(\R^n)$ of the homogeneous problem \eqref{eqlinear} behaves like $d^s$ (by Theorem \ref{thGrubb}).

Note, therefore, that \cite{Grubb} treats both problems \eqref{eqlinear} and \eqref{eqlinear2} ---while here we only deal with $H^s(\R^n)$ solutions to \eqref{eqlinear}.
Notice also that when $s=1$ \eqref{eqlinear2} is the Dirichlet problem for $\Delta$ with boundary conditions $u=\varphi$ on $\partial\Omega$, while for $s=2$ it is the Dirichlet problem for $\Delta^2$ with boundary conditions $u=0$ and $\partial_\nu u=\varphi$ on $\partial\Omega$.
\end{rem}

The paper is organized as follows.
In Section \ref{sec-ingredients} we present the main ingredients of the proof of Theorem \ref{intparts}.
In Section \ref{sec3} we show Proposition \ref{thlaps/2}.
In Section \ref{sec2} we prove Theorem \ref{intparts} for strictly star-shaped domains.
In Section \ref{sec4} we finish the proof of Theorem \ref{intparts} and we prove Theorem \ref{corintparts} and Corollaries \ref{thpoh} and \ref{cor-uniquecont}.
Finally, in the Appendix we prove Lemma \ref{computation}.

\section{Ingredients of the proof}
\label{sec-ingredients}

The formal ideas of the proof of the Pohozaev identity for $s>1$ are essentially the same as in the case $s\in(0,1)$ \cite{RS-Poh}.
However, some parts must be carefully adapted, especially when obtaining the singular behavior of $(-\Delta)^{s/2}u$ near $\partial\Omega$, given by Proposition \ref{thlaps/2} below.
We will skip the details of the parts that are similar to the case $s\in(0,1)$, to focus in the parts that present new mathematical difficulties.

To prove Theorem \ref{intparts} we proceed as follows.
We first assume the domain $\Omega$ to be star-shaped with respect to the origin and $(-\Delta)^s u$ to be smooth.
The result for general smooth domains follows from this star-shaped case, as in our proof of the case $s\in(0,1)$.
Moreover, the result for functions $u$ satisfying only $u\in H^s(\R^n)$ and $(-\Delta)^su\in L^\infty(\Omega)$ follows by approximation using Theorem \ref{thGrubb}.

When the domain is star-shaped, one shows that
\[\int_\Omega(x\cdot\nabla u)(-\Delta)^su\ dx=\frac{2s-n}{2}\int_{\Omega}u(-\Delta)^s u\,dx+\frac12\left.\frac{d}{d\lambda}\right|_{\lambda=1^+}\int_{\R^n} w_{{\lambda}} w_{1/{\lambda}}dy,\]
where
\[w(x)=(-\Delta)^{s/2}u(x)\quad \textrm{and}\quad w_\lambda(x)=w(\lambda x).\]
Then, Theorem \ref{intparts} is equivalent to the following identity
\begin{equation}\label{partdificil}
-\left.\frac{d}{d\lambda}\right|_{\lambda=1^+}\int_{\mathbb R^n} w_{{\lambda}} w_{1/{\lambda}}\ dy=
\Gamma(1+s)^2\int_{\partial\Omega}\left(\frac{u}{d^s}\right)^2(x\cdot\nu)d\sigma.
\end{equation}
The quantity $\frac{d}{d\lambda}|_{\lambda=1^+}\int_{\R^n}w_{\lambda}w_{1/\lambda}$ vanishes for any $C^1(\R^n)$ function $w$, as can be seen by differentiating under the integral sign.
Instead, we will see that, under the hypotheses of Theorem \ref{intparts}, the function $w=(-\Delta)^{s/2} u$ has a logarithmic singularity along the boundary $\partial\Omega$, and that \eqref{partdificil} holds.

The identity \eqref{partdificil} is the difficult part of the proof of Theorem \ref{intparts}.
To establish it, it is crucial to know the precise behavior of $(-\Delta)^{s/2}u$ near $\partial\Omega$ ---from both inside and outside $\Omega$.
This is given by the following result.

\begin{prop}\label{thlaps/2}
Let $\Omega$ be a bounded and smooth domain, and let $d$ be as in Definition \ref{d}.
Let $u\in H^s(\R^n)$ be a function such that $u\equiv0$ in $\mathbb R^n\backslash\Omega$ and that $(-\Delta)^su$ is $C^\infty(\overline \Omega)$.

Let $\delta(x)=\textrm{dist}(x,\partial\Omega)$ ---notice that $d\equiv0$ in $\R^n\setminus\overline\Omega$, while $\delta>0$ therein.

Then, there is a $C^\infty(\R^n)$ extension $v$ of $u/d^s|_{\Omega}$ such that
\begin{equation}\label{**}
(-\Delta)^{s/2}u(x)=c_1\left\{\log^-\delta(x)+c_2\chi_{\Omega}(x)\right\}v(x)+h(x)\quad \textrm{in } \ \R^n,
\end{equation}
where $h$ is a $C^\alpha(\R^n)$ function for some $\alpha>0$, $\log^-t=\min\{\log t,0\}$, and $c_1$ and $c_2$ are constants that depend only on $s$.

Moreover, for all $\beta\in(0,+\infty)$,
\begin{equation}\label{numeret}
\left[(-\Delta)^{s/2}u\right]_{C^{\beta}(\{x\in\R^n:\,d(x)\ge\rho\})}\leq C\rho^{-\beta}\qquad \textrm{for all}\ \  \rho\in(0,1),
\end{equation}
for some constant $C$ which does not depend on $\rho$.
\end{prop}

The constant $c_2$ comes from an involved expression and it is nontrivial to compute.
Instead of computing it explicitly, we compute directly the constant $\Gamma(1+s)^2$ in Theorem \ref{intparts} by using an
explicit solution $u$ to problem \eqref{eq} in a ball.
This explicit solution is given by the following lemma.

\begin{lem}\label{computation}
Let $s>1$ be any noninteger number.
Then, the solution to problem
\[\left\{ \begin{array}{rcll}
(-\Delta)^s u &=&1&\textrm{in }B_1(0) \\
u&=&0&\textrm{in }\mathbb R^n\backslash B_1(0)
\end{array}\right.\]
is given by
\[u(x)=\frac{2^{-2s}\Gamma(n/2)}{\Gamma\left(\frac{n+2s}{2}\right)\Gamma(1+s)}\left(1-|x|^2\right)^s\qquad\textrm{in}\ \ B_1(0),\]
where $\Gamma$ is the Gamma function.
\end{lem}

For $s\in(0,1)$, this expression was obtained by Getoor \cite{G} in 1961.
Here, we show that the same expression holds for $s>1$ by using some explicit computations of Dyda \cite{Dyda} ---see the proof in the Appendix.

\section{Behavior of $(-\Delta)^{s/2}u$ near $\partial\Omega$}
\label{sec3}

The aim of this section is to prove Proposition \ref{thlaps/2}.
We will split its proof into three partial results.
The first one is the following, and compares the behavior of $(-\Delta)^{s/2}u$ near $\partial\Omega$ with the one of $(-\Delta)^{s/2}d^s$.

\begin{prop}\label{proplaps2}
Let $\Omega$ be a bounded and smooth domain, $u$ be a function satisfying the hypotheses of Proposition \ref{thlaps/2}, and $d$ be as in Definition \ref{d}.
Then, there exists a $C^\infty(\R^n)$ extension $v$ of $u/d^s|_\Omega$ such that
\[(-\Delta)^{s/2}u(x)=v(x)(-\Delta)^{s/2}d^s(x)+h(x)\ \mbox{ in }\ \R^n,\]
where $h\in C^\alpha(\R^n)$ for some $\alpha>0$.
\end{prop}

The following result gives the behavior of $(-\Delta)^s d^s$ near $\partial\Omega$.

\begin{prop}\label{prop:Lap-s/2-delta-s}
Let $\Omega$ be a bounded and smooth domain.
Let $d$ be as in Definition \ref{d}, and $\delta(x)=\textrm{dist}(x,\partial\Omega)$.
Then,
\[(-\Delta)^{s/2}d^s(x)=c_1\left\{\log^- \delta(x)+c_2\chi_\Omega(x)\right\}+h(x)\ \mbox{ in }\ \R^n,\]
where $c_1$ and $c_2$ are constants, $h$ belongs to $C^\alpha(\R^n)$, and $\log^-t=\min\{\log t,0\}$.
Moreover, the constants $c_1$ and $c_2$ depend only on $s$.
\end{prop}

Recall that, in the previous result, $\delta(x)>0$ in $\R^n\setminus\overline\Omega$, while $d(x)\equiv0$ therein.

The third one gives estimate \eqref{numeret} in Proposition \ref{thlaps/2}, and reads as follows.
It will be also used in the proof of Proposition \ref{prop:Lap-s/2-delta-s}.

\begin{lem}\label{prop12345}
Let $\Omega\subset\R^n$ be a bounded and smooth domain.
Let $u\in H^s(\R^n)$ be a function such that $u\equiv0$ in $\mathbb R^n\backslash\Omega$ and that $(-\Delta)^su$ is $C^\infty(\overline \Omega)$.
Then, for all $\beta\in(0,+\infty)$,
\begin{equation}\label{numeretb}
\left[(-\Delta)^{s/2}u\right]_{C^{\beta}(\{{\rm dist}(x,\partial\Omega)\ge\rho\})}\leq C\rho^{-\beta}\qquad \textrm{for all}\ \  \rho\in(0,1),
\end{equation}
for some constant $C$ that does not depend on $\rho$.
\end{lem}

In the proof of Proposition \ref{proplaps2} we need to compute $(-\Delta)^{s/2}$ of the product $u=d^s v$.
For it, we will use the following identity for $s_0\in(0,1)$
\[(-\Delta)^{s_0}(w_1w_2)=w_1(-\Delta)^{s_0}w_2+w_2(-\Delta)^{s_0}w_1-I_{s_0}(w_1,w_2),\]
where
\begin{equation}\label{Is}
I_{s_0}(w_1,w_2)(x)= c_{n,s_0}\textrm{PV}\int_{\R^n} \frac{\bigl(w_1(x)-w_1(y)\bigr)\bigl(w_2(x)-w_2(y)\bigr)}{|x-y|^{n+2s_0}}\,dy;
\end{equation}
see \cite{RS}.
Another ingredient in the proof of Proposition \ref{proplaps2} is the following result.

\begin{lem}\label{s>2}
Let $\Omega\subset\R^n$ be any bounded smooth domain.
Let $d$ be as in Definition \ref{d}.
Let $k$ be a positive integer, and let $\beta>k$ be a noninteger.
Then,
\[\Delta^k d^\beta = \beta(\beta-1)\cdots(\beta-2k+1)d^{\beta-2k}+h,\]
for some function $h\in C^{\beta-2k+1}(\R^n)$.
\end{lem}

\begin{proof}
We prove the result for $k=1$.
For $k\geq2$, it follows immediately by induction.

For $k=1$, we use the following geometric formula for the Laplacian
\[\Delta \phi(x_0)=\partial_{\nu\nu}\phi(x_0)+H_{\Gamma}(x_0)\partial_\nu\phi(x_0)+\Delta_{\Gamma}\phi(x_0),\]
applied to the level set $\Gamma=\{\phi(x)=\phi(x_0)\}$.
In this formula, $\nu$ is the unit outward tangent to $\Gamma$ at $x_0$, $H_{\Gamma}$ is the mean curvature of $\Gamma$, and $\Delta_\Gamma$ is the Laplace-Beltrami operator on $\Gamma$.
From this formula, we deduce that
\[\Delta \bigl(d^\beta\bigr)=\beta(\beta-1)d^{\beta-2}+H_{\Gamma} d^{\beta-1}.\]
Since the level sets of $d$ are smooth, then $H_{\Gamma}$ is $C^\infty$.
Hence, $h=H_{\Gamma}d^{\beta-1}$ belongs to $C^{\beta-1}$, and thus the result is proved.
\end{proof}

We next give the

\begin{proof}[Proof of Proposition \ref{proplaps2}]
Let $k$ be the integer number for which $s/2=s_0+k$, with $s_0\in(0,1)$.
Since $\Omega$ is $C^\infty$ and $u/d^s$ is $C^\infty(\overline\Omega)$, then there exists a $C^\infty(\R^n)$ extension of $u/d^s|_\Omega$, that we denote by $v$.
Then,
\[\begin{split}
(-\Delta)^{s/2}u&=(-\Delta)^{s/2}\left(v d^s\right)\\
&=(-\Delta)^{s_0}(-\Delta)^k\left(v d^s\right)\\
&=(-\Delta)^{s_0}\left\{d^s(-\Delta)^kv+\cdots+v(-\Delta)^kd^s\right\}\\
&=(-\Delta)^{s_0}\left\{v(-\Delta)^kd^s+h_1\right\},
\end{split}.\]
where $h_1$ is a $C^{s-2k+1}(\R^n)$ function (since $d^s$ is $C^s$ and we differentiate it up to $2k-1$ times).
Since $s-2k=2s_0$, then $h_2=(-\Delta)^{s_0}h_1$ is $C^\alpha(\R^n)$ for all $\alpha<1$.

Thus,
\[(-\Delta)^{s/2}u=(-\Delta)^{s_0}\left\{v(-\Delta)^kd^s\right\}+h_2,\]
with $h_2\in C^\alpha(\R^n)$.
By Lemma \ref{s>2}, we have
\begin{equation}\label{klm}
(-\Delta)^kd^s=c_sd^{2s_0}+h_3
\end{equation}
for some function $h_3\in C^{2s_0+1}$ and some constant $c_s$ that depend only on $s$.
Thus,
\begin{equation}\label{klm2}
(-\Delta)^{s/2}u=c_s(-\Delta)^{s_0}\left\{vd^{2s_0}\right\}+h_4
\end{equation}
for some $h_4\in C^\alpha(\R^n)$.

Now, since $s_0\in(0,1)$, we have
\[(-\Delta)^{s_0}\left\{vd^{2s_0}\right\}=v(-\Delta)^{s_0}d^{2s_0}+d^{2s_0}(-\Delta)^{s_0}v-I_{s_0}(v,d^{2s_0}),\]
where $I_{s_0}$ is given by \eqref{Is}.
Since $v\in C^\infty(\R^n)$ and $d^{2s_0}\in C^{2s_0}(\R^n)$, then the second and the third term belong to $C^\alpha(\R^n)$.
Therefore, we have
\[(-\Delta)^{s_0}\left\{vd^{2s_0}\right\}=v(-\Delta)^{s_0}d^{2s_0}+h_5,\]
where $h_5\in C^\alpha(\R^n)$.
Finally, using \eqref{klm} and \eqref{klm2}, we find
\[(-\Delta)^{s/2}u=v (-\Delta)^{s/2}d^s+h_6\]
for some function $h_6\in C^\alpha(\R^n)$, as desired.
\end{proof}

To prove Proposition \ref{prop:Lap-s/2-delta-s} we will need some lemmas.

Fixed $\rho_0>0$, let $\phi$ be any \emph{bounded} function in $C^{s}(\R)\cap C^{\infty}(\R_+)$ such that
\begin{equation}\label{phi}
\phi(x)= (x_+)^s \qquad \textrm{in}\quad (-\infty,\rho_0).
\end{equation}
This function $\phi$ coincides with the $s$-harmonic function $(x_+)^s$ in a neighborhood of the origin, but $\phi$ is bounded at infinity.
We need to introduce it because the growth at infinity of $(x_+)^s$ prevents us
from computing its $(-\Delta)^{s/2}$.

\begin{lem}\label{claim1} Let $\rho_0>0$, and let $\phi:\R\rightarrow\R$ be given by \eqref{phi}.
Then, we have
\[(-\Delta)^{s/2}\phi(x)= c_1\bigl\{\log |x| + c_2\chi_{(0,\infty)}(x)\bigr\} + h(x)\]
for $x\in(-\rho_0/2,\rho_0/2)$, where $h\in {\rm Lip}([-\rho_0/2,\rho_0/2])$.
The constants $c_1$ and $c_2$ depend only on $s$.
\end{lem}

\begin{proof}
If $s>2$, then one can write $s/2=s_0/2+k$ to find that the function $\tilde\phi=(-\Delta)^k \phi$ is $C^{\infty}(\R_+)$ and satisfies $\tilde\phi=c_s(x_+)^{s_0}$ in $(-\infty,\rho_0)$ for some constant $c_s$.
Hence, the case $s>2$ follows from the case $s\in(0,2)$.

The case $s\in(0,1)$ was done in \cite[Lemma 3.7]{RS-Poh}.
Thus, from now on we assume $s\in(1,2)$.
Then, we have
\[(-\Delta)^{s/2}\phi(x)=\frac{c_{1,s}}{2}\int_{-\infty}^{+\infty}\frac{2\phi(x)-\phi(x+y)-\phi(x-y)}{|y|^{1+s}}dy.\]
Hence, for $x\in(-\rho_0/2,\rho_0/2)$,
\[(-\Delta)^{s/2}\phi(x)=\frac{c_{1,s}}{2}\int_{x-\rho_0}^{\rho_0-x}\frac{2x_+^s-(x+y)_+^s-(x-y)_+^s}{|y|^{1+s}}dy+\phi_0(x),\]
where $\phi_0$ is a smooth function in $[-\rho_0/2,\rho_0/2]$.

Let us study
\[\begin{split}J(x)&=\int_{x-\rho_0}^{\rho_0-x}\frac{2x_+^s-(x+y)_+^s-(x-y)_+^s}{|y|^{1+s}}dy\\
&=\begin{cases}\displaystyle
\ J_1(x)= \int_{1-\frac{\rho_0}{x}}^{\frac{\rho_0}{x}-1}\frac{2-(1+z)_+^s-(1-z)_+^s}{|z|^{1+s}}\,dz &\mbox{if }x>0\\
&\\
\displaystyle
\ J_2(x)=\int_{-1-\frac{\rho_0}{|x|}}^{\frac{\rho_0}{|x|}+1}\frac{-(-1+z)_+^s-(-1-z)_+^s}{|z|^{1+s}} \,dz&     \mbox{if }x<0.
\end{cases}\end{split}\]

Using L'H\^opital's rule we easily find that
\[\lim_{x\downarrow0} \frac{J_1(x)} {\log |x|}=\lim_{x\uparrow0}
\frac{J_2(x)}{\log |x|}=2.\]
Moreover,
\[\begin{split}
\lim_{x\downarrow0} \left\{ J_1'(x)-\frac{2}{x}\right\}&=
\lim_{x\downarrow0} \left\{ -\frac{\rho_0}{x^2}\frac{2-\left(\frac{\rho_0}{x}\right)^s}{\left(\frac{\rho_0}{x}-1\right)^{1+s}}-
\frac{\rho_0}{x^2}\frac{2-\left(\frac{\rho_0}{x}\right)^s}{\left(\frac{\rho_0}{x}-1\right)^{1+s}}-\frac{2}{x}\right\}\\
&=\frac{2}{\rho_0^s}\lim_{y\downarrow0} \left\{ \frac{1}{y^2}\frac{2-\left(\frac{1}{y}\right)^s}{\left(\frac{1}{y}-1\right)^{1+s}}-\frac{1}{y}\right\} \\
&=\frac{2}{\rho_0^s}\lim_{y\downarrow0} \left\{ \frac{1-2y^s-(1-y)^{1+s}}{y\left(1-y\right)^{1+s}}\right\}=\frac{2(1+s)}{\rho_0^s}.
\end{split}\]
Thus, $J_1(x)-2\log|x|$ is Lipschitz in $[0,\rho_0/2]$.

Analogously,
\[\begin{split}
\lim_{x\uparrow0} \left\{ J_2'(x)+\frac{2}{x}\right\}&=
\lim_{x\uparrow0} \left\{ 2\frac{\rho_0}{x^2}\frac{-\left(-\frac{\rho_0}{x}\right)^s}{\left(-\frac{\rho_0}{x}-1\right)^{1+s}}
+\frac{2}{x}\right\}\\
&=\frac{2}{\rho_0^s}\lim_{y\downarrow0} \left\{ \frac{1}{y^2}\frac{-\left(\frac{1}{y}\right)^s}{\left(\frac{1}{y}-1\right)^{1+s}}+\frac{1}{y}\right\} \\
&=\frac{2}{\rho_0^s}\lim_{y\downarrow0} \left\{ \frac{-1+(1-y)^{1+s}}{y\left(1-y\right)^{1+s}}\right\}=-\frac{2(1+s)}{\rho_0^s}.
\end{split}\]
Thus, the function $J_2(x)-2\log|x|$ is Lipschitz in $[-\rho_0/2,0]$.

Therefore, for some appropriate constant $c$, the function $J(x)-2\log|x|-c\chi_{(0,\infty)}(x)$ is Lipschitz in $[-\rho_0/2,\rho_0/2]$, and the lemma is proved.
\end{proof}

Next lemma will be used to prove Proposition \ref{prop:Lap-s/2-delta-s}.
Before stating it, we need the following

\begin{rem}\label{remrho0}
From now on in this section, $\rho_0>0$ is a small constant depending only on $\Omega$ such that that every point on $\partial\Omega$ can be touched from both inside and outside $\Omega$ by balls of radius $\rho_0$.
A useful observation is that all points $y$ in the segment that joins $x_1$ and $x_2$ ---through $x_0$--- satisfy
$d(y)= |y-x_0|$.
\end{rem}

\begin{lem}\label{claim2}
Assume $s\in(1,2)$.
Let $\Omega$ be a bounded smooth domain, and $\rho_0$ be given by Remark \ref{remrho0}.
Fix $x_0\in\partial\Omega$, and let
\[\phi_{x_0}(x)=\phi\left(-\nu(x_0)\cdot(x-x_0)\right)\]
and
\begin{equation}\label{Sx0}
S_{x_0}=\bigl\{x_0+t\nu(x_0),\ t\in(-\rho_0/2,\rho_0/2)\bigr\},
\end{equation}
where $\phi$ is given by \eqref{phi} and $\nu(x_0)$ is the unit outward normal to $\partial\Omega$ at $x_0$.
Define also $w_{x_0}=d^s-\phi_{x_0}$.

Then, for all $x\in S_{x_0}$,
\[\left|(-\Delta)^{s/2} w_{x_0}(x)-(-\Delta)^{s/2} w_{x_0}(x_0)\right| \le C|x-x_0|^{1-\frac{s}{2}},\]
where $C$ depends only on $\Omega$ and $\rho_0$ (and not on $x_0$).
\end{lem}

\begin{proof}
We follow \cite[Lemma 3.9]{RS-Poh}, where a very similar result was proved for the case $s\in(0,1)$.

We denote $w=w_{x_0}$.
Note that, along $S_{x_0}$, the distance to $\partial\Omega$ agrees with the distance to the tangent plane to $\partial\Omega$ at $x_0$; see Remark \ref{remrho0}.
That is, denoting $d_\pm=(\chi_\Omega-\chi_{\R^n\backslash \Omega})d$ and $\tilde\delta(x)=-\nu(x_0)\cdot(x-x_0)$, we have $d_\pm(x)=\tilde\delta(x)$ for all $x\in S_{x_0}$.
Moreover, the gradients of these two functions also coincide on $S_{x_0}$, i.e., $\nabla d_\pm(x)=-\nu(x_0)=\nabla \tilde\delta(x)$ for all $x\in S_{x_0}$.

Therefore, for all $x\in S_{x_0}$ and $y\in B_{\rho_0/2}(0)$, we have
\[|d_{\pm}(x+y)-\tilde\delta(x+y)|\le C|y|^2\]
for some $C$ depending only on $\rho_0$.
Thus, since $s>1$, for all $x\in S_{x_0}$ and $y\in B_{\rho_0/2}(0)$ we have
\begin{equation}\label{w(x+y)}|w(x+y)|= |(d_\pm(x+y))_+^s-(\tilde\delta(x+y))_+^s| \le C|y|^2,\end{equation}
where $C$ is a constant depending on $\Omega$ and $s$.

On the other hand, since $w$ is Lipschitz and bounded,
\begin{equation}\label{w(x+y)2}|w(x+y)-w(x_0+y)|\leq C|x-x_0|.\end{equation}

Finally, let $r<\rho_0/2$ to be chosen later.
For each $x\in S_{x_0}$ we have
\[
\begin{split}
\bigl|(-\Delta)^{s/2} &w(x)-(-\Delta)^{s/2} w(x_0)\bigr|
\le C\int_{\R^n} \frac{|w(x+y)-w(x_0+y)|}{|y|^{n+s}}\,dy\\
&\le C\int_{B_r}\frac{|w(x+y)-w(x_0+y)|}{|y|^{n+s}}\,dy
+C\int_{\R^n\setminus B_r}\frac{|w(x+y)-w(x_0+y)|}{|y|^{n+s}}\,dy\\
&\le C\int_{B_r}\frac{|y|^2}{|y|^{n+s}}\,dy+
C\int_{\R^n\setminus B_r}\frac{|x-x_0|}{|y|^{n+s}}\,dy\\
&= C\left(r^{2-s} +|x-x_0|r^{-s}\right)\,,
\end{split}
\]
where we have used \eqref{w(x+y)} and \eqref{w(x+y)2}.
Taking $r=|x-x_0|^{1/2}$ the lemma is proved.
\end{proof}

To prove Proposition \ref{prop:Lap-s/2-delta-s} we will use the following.

\begin{claim}[\cite{RS-Poh}]\label{claimMMM}
Let $\Omega$ be a bounded $C^{1,1}$ domain, and $\rho_0$ be given by Remark \ref{remrho0}.
Let $\delta(x)=\textrm{dist}(x,\partial\Omega)$.
Let $w$ be a function satisfying, for some $K>0$,
\begin{itemize}
\item[(i)] $w$ is locally Lipschitz in $\{x\in\R^n \,:\,0<\delta(x)<\rho_0\}$ and
\[|\nabla w(x)| \le K\delta(x)^{-M}\ \mbox{ in }\ \{x\in\R^n \,:\,0<\delta(x)<\rho_0\}\]
for some $M >0$.
\item[(ii)] There exists $\alpha>0$ such that
\[|w(x)-w(x^*)|\le K\delta(x)^\alpha \ \mbox{ in }\ \{x\in\R^n \,:\,0<\delta(x)<\rho_0\},\]
where $x^*$ is the unique point on $\partial \Omega$ satisfying $\delta(x)= |x-x^*|$.
\item[(iii)] For the same $\alpha$, it holds
\[
\|w\|_{C^{\alpha}\left( {\{\delta\geq\rho_0\}}\right)}\le K.
\]
\end{itemize}

Then, there exists $\gamma>0$, depending only on $\alpha$ and $M$, such that
\begin{equation}\label{claimMMMeq}
\|w\|_{C^{\gamma}(\R^n)}\le C K,
\end{equation}
where $C$ depends only on $\Omega$.
\end{claim}

Finally, we give the

\begin{proof}[Proof of Proposition \ref{prop:Lap-s/2-delta-s}]
For $s>2$, write $s/2=s_1/2+k$, with $k\in \mathbb Z$ and $s_1\in(0,2)$, and use Lemma \ref{s>2}.
We find $(-\Delta)^{s/2} d^s=c(-\Delta)^{s_1/2}d^{s_1}+h_0$, where $h_0\in  C^\alpha(\R^n)$.
Hence, we may assume $s\in(0,2)$ in the rest of the proof.
Moreover, the case $s\in(0,1)$ was done in \cite[Proposition 3.2]{RS-Poh}.
Thus, from now on we assume $s\in(1,2)$.

For $s\in(1,2)$, we will use the previous results and Claim \ref{claimMMM} to prove the result.
Define
\[h(x)= (-\Delta)^{s/2}d^s(x)-c_1\left\{\log^- \delta(x)+c_2\chi_\Omega(x)\right\},\]
where $c_1$ and $c_2$ are given by Lemma \ref{claim1}.

On one hand, by Lemma \ref{claim1}, for all $x_0\in\partial\Omega$ and for all $x\in S_{x_0}$, where $S_{x_0}$ is defined by \eqref{Sx0}, we have
\[h (x) = (-\Delta)^{s/2}d^s(x)- (-\Delta)^{s/2} \phi_{x_0}(x) + \tilde h\bigl(\nu(x_0)\cdot(x-x_0)\bigr),\]
where $\tilde h$ is the $\textrm{Lip}([-\rho_0/2,\rho_0/2])$ function given by Lemma \ref{claim1}.
Hence, using Lemma \ref{claim2}, we find
\[|h(x)-h(x_0)|\le C|x-x_0|^{\alpha}\quad \mbox{for all }x\in S_{x_0}\]
for some constant independent of $x_0$.

Recall that for all $x\in S_{x_0}$ we have $x^*=x_0$, where $x^*$ is the unique point on $\partial \Omega$ satisfying $d(x)= |x-x^*|$.
Hence,
\begin{equation}\label{eq:pf-prop-calpha-bound1}
|h(x)-h(x^*)|\le C|x-x^*|^{\alpha}\quad \mbox{ for all }x\in\{d(x)<\rho_0/2\}\,.
\end{equation}

Moreover, since $d^s$ is $C^\infty$ inside $\Omega$, then
\begin{equation}\label{eq:pf-prop-bound-faraway}
\|h\|_{C^{\alpha}(\{\delta\ge\rho_0/2\})} \le C.
\end{equation}

Now, if $0<\delta(x)<\rho_0/2$, then
\begin{equation}\label{eq:pf-prop-grad-bound}
|\nabla h(x)|\le |\nabla(-\Delta)^{s/2}d^s(x)| + c_1|d(x)|^{-1}\le C|\delta(x)|^{-1}.
\end{equation}
The last inequality follows from Proposition \ref{prop12345} (with $\beta=1$), since $(-\Delta)^sd^s(x)$ is smooth (by Theorem \ref{thGrubb}).

Therefore, using \eqref{eq:pf-prop-calpha-bound1}, \eqref{eq:pf-prop-bound-faraway}, and \eqref{eq:pf-prop-grad-bound} and Claim \ref{claimMMM}, we find that $h\in C^\alpha(\R^n)$ for some small $\alpha>0$, as desired.
\end{proof}

We next prove Proposition \ref{prop12345}.

\begin{proof}[Proof of Proposition \ref{prop12345}]
First, clearly we may assume that $s<2$ ---otherwise, we apply the result to $-\Delta u$.
Since the case $s\in(0,1)$ was done in \cite[Lemma 4.3]{RS-Poh}, from now on we assume that $s\in(1,2)$.

Note that, since $(-\Delta)^su$ is smooth, then by Theorem \ref{thGrubb} $u/d^s$ is smooth in $\overline \Omega$, and thus $u$ satisfies
\[[u]_{C^{\gamma+s}(\{\textrm{dist}(x,\partial\Omega)\geq \rho\})}\leq C\rho^{-\gamma}\]
for all $\gamma\geq-s$.

Take $x_0\in\Omega$, and define $R>0$ so that $\textrm{dist}(x_0,\partial\Omega)=2R$.
We will prove that
\[[(-\Delta)^{s/2}u]_{C^\beta(B_{R/2}(x_0))}\leq CR^{-\beta},\]
which yields the desired result (see \cite{RS} for more details).

As in the proof of Lemma 4.3 in \cite{RS}, we may write $u=u(x_0)+\bar u+\varphi$, where $\bar u$ and $\varphi$ satisfy
\[\varphi\equiv0\quad\textrm{in}\ B_R(x_0),\qquad \bar u\equiv0\quad\textrm{outside}\ B_{3R/2}(x_0),\]
and
\begin{equation}\label{cotes-interiors}
[\bar u]_{C^{\gamma+s}(\R^n)}\leq CR^{-\gamma}\quad \textrm{for all}\ \gamma\geq -s.
\end{equation}
As in \cite[Lemma 4.3]{RS}, using the regularity of the kernel of the fractional Laplacian, that $\varphi\equiv0$ in $B_R(x_0)$, we find that
\[[(-\Delta)^{s/2}\varphi]_{C^\beta(B_{R/2}(x_0))}\leq CR^{-\beta}.\]
Thus, we only have to show that for all $x_1$ and $x_2$ in $B_{R/2}(x_0)$,
\begin{equation}\label{we-have-to}
\left|D^k(-\Delta)^{s/2}\bar u(x_1)-D^k(-\Delta)^{s/2}\bar u(x_2)\right|\leq C|x_1-x_2|^{\beta'}R^{-k-\beta'},
\end{equation}
where $\beta=\beta'+k$, with $k$ integer and $\beta'\in(0,1]$.
Here, $D^k$ denotes any $k$-th order derivative.

We prove \eqref{we-have-to} for the case $\beta'=1$, since the other ones clearly follow from it (using $|x_1-x_2|\leq R$).
Denote
\[w=D^k\bar u\quad \textrm{and}\quad \delta^2w(x,y)=\frac{1}{2}\bigl(2w(x)-w(x+y)-w(x-y)\bigr).\]
Using that $w$ has support in $B_{3R/2}(x_0)$, and also \eqref{cotes-interiors} with $\gamma+s=k+3$ and with $\gamma+s=k+1$, we find that for all $x\in B_{R/2}(x_0)$
\[\begin{split}
|\nabla(-\Delta)^{s/2}w(x)|&\leq C\int_{B_{2R}(x_0)}|\delta^2\nabla w(x,y)|\frac{dy}{|y|^{n+s}}+C\int_{\R^n\setminus B_{2R}(x_0)}|\nabla w(x)|\frac{dy}{|y|^{n+s}}\\
&\leq C\int_{B_{2R}(x_0)}C|y|^2R^{s-k-3}\frac{dy}{|y|^{n+s}}+C\int_{\R^n\setminus B_{2R}(x_0)}R^{-k-1+s}\frac{dy}{|y|^{n+s}}\\
&\leq CR^{-k-1}.\end{split}\]
Thus, for all $x_1$ and $x_2$ in $B_{R/2}(x_0)$,
\[\left|(-\Delta)^{s/2}w(x_1)-(-\Delta)^{s/2}w(x_2)\right|\leq C|x_1-x_2|R^{-k-1}.\]
Therefore, \eqref{we-have-to} follows, and the Proposition is proved.
\end{proof}

Finally, we give the

\begin{proof}[Proof of Proposition \ref{thlaps/2}]
It immediately follows from Propositions \ref{proplaps2}, \ref{prop:Lap-s/2-delta-s}, and Lemma \ref{prop12345}.
\end{proof}

\section{Case I: Star-shaped domains}
\label{sec2}

In this section we prove Theorem \ref{intparts} for strictly star-shaped domains and for $(-\Delta)^su$ smooth in $\overline\Omega$.
Recall that $\Omega$ is said to be strictly star-shaped if, for some $z_0\in\R^n$,
\begin{equation}\label{starshaped}
(x-z_0)\cdot\nu>0\qquad \textrm{for all}\ \ x\in\partial\Omega.
\end{equation}

A key tool in this proof is Proposition 1.11 of \cite{RS-Poh}, which is stated next.

\begin{prop}[\cite{RS-Poh}]\label{propoperador}
Let $A$ and $B$ be real numbers, and
\[\varphi(t)=A\log^-|t-1|+B\chi_{[0,1]}(t)+h(t),\]
where $\log^-t=\min\{\log t,0\}$ and $h$ is a function satisfying, for some constants $\alpha$ and $\gamma$ in $(0,1)$, and $C_0>0$, the following conditions:
\begin{itemize}
\item[(i)] $\|h\|_{C^{\alpha}([0,\infty))}\leq C_0$.
\item[(ii)] For all $\beta\in[\gamma,1+\gamma]$
\[\|h\|_{C^{\beta}((0,1-\rho)\cup(1+\rho,2))}\leq C_0 \rho^{-\beta}\qquad \textrm{for all}\ \  \rho\in(0,1).\]
\item[(iii)] $|h'(t)|\leq C_0 t^{-2-\gamma}$ and $|h''(t)|\leq C_0 t^{-3-\gamma}$ for all $t>2$.
\end{itemize}
Then,
\[-\left.\frac{d}{d\lambda}\right|_{\lambda=1^+}\int_{0}^{\infty} \varphi\left(\lambda t\right)\varphi\left(\frac{t}{\lambda}\right)dt=A^2\pi^2+B^2.\]

Moreover,  the limit defining this derivative is uniform among functions $\varphi$ satisfying (i)-(ii)-(iii) with given constants $C_0$, $\alpha$, and $\gamma$.
\end{prop}

We proceed now with the proof of Theorem \ref{intparts}.

\begin{proof}[Proof of Theorem \ref{intparts} for strictly star-shaped domains and for $(-\Delta)^su\in C^\infty(\overline\Omega)$]
As explained in the Introduction, we follow the proof of Proposition 1.6 in \cite{RS-Poh}.

Recall that we are assuming $\Omega$ to be strictly star shaped.
We may assume without loss of generality that $\Omega$ is strictly star-shaped with respect to the origin, that is, $z_0=0$ in \eqref{starshaped}.
Indeed, if $\Omega$ is strictly star-shaped with respect to a point $z_0\neq0$, then the result follows from the case $z_0=0$, as in \cite[Proposition 1.6]{RS-Poh}.

By Theorem \ref{thGrubb}, we have that $u\in C^s(\R^n)$, and thus in particular $u\in C^1(\R^n)$.
Thus, by the dominated convergence theorem,
\begin{equation}\label{first}
\int_\Omega(x\cdot \nabla u) (-\Delta)^su\, dx=\left.\frac{d}{d\lambda}\right|_{\lambda=1^+}\int_\Omega u_\lambda (-\Delta)^su\, dx,
\end{equation}
where $\left.\frac{d}{d\lambda}\right|_{\lambda=1^+}$ is the derivative from the right side at $\lambda=1$.

Now, as in \cite[Proposition 1.6]{RS-Poh}, integrating by parts and making a change of variables we find
\[\int_\Omega u_\lambda (-\Delta)^su\,dx=\lambda^{\frac{2s-n}{2}}\int_{\mathbb R^n}w_{\sqrt{\lambda}}\,w_{1/\sqrt{\lambda}}\,dy,\]
where
\[w(x)=(-\Delta)^{s/2}u(x)\qquad {\rm and}\qquad w_{\lambda}(x)=w(\lambda x).\]
Thus,
\begin{equation}\label{igualtatpaluego}
\int_\Omega(\nabla u\cdot x)(-\Delta)^su\, dx=\frac{2s-n}{2}\int_{\Omega}u(-\Delta)^s u\, dx+
\frac 12\left.\frac{d}{d\lambda}\right|_{\lambda=1^+}\int_{\mathbb R^n}w_{{\lambda}}w_{1/{\lambda}}\,dy.
\end{equation}
Hence, Theorem \ref{intparts} is equivalent to the identity
\begin{equation}\label{derivadaIlambda}
-\left.\frac{d}{d\lambda}\right|_{\lambda=1^+}I_\lambda=\Gamma(1+s)^2\int_{\partial\Omega}\left(\frac{u}{d^s}\right)^2(x\cdot\nu)\,d\sigma,
\end{equation}
where
\begin{equation}\label{Ilambda}
I_\lambda=\int_{\mathbb R^n}w_{{\lambda}}w_{1/{\lambda}}\,dy.
\end{equation}

Now, as in \cite{RS-Poh}, since the domain is strictly star-shaped we can write the integral $I_\lambda$ in spherical coordinates to obtain
\[\left.\frac{d}{d\lambda}\right|_{\lambda=1^+}I_\lambda=
\left.\frac{d}{d\lambda}\right|_{\lambda=1^+}\int_{\partial\Omega}(x\cdot \nu)d\sigma(x)
\int_0^\infty t^{n-1}w(\lambda tx)w\left(\frac{tx}{\lambda}\right)dt.\]

Fix now $x_0\in \partial\Omega$, and define
\[\varphi(t)=t^{\frac{n-1}{2}}w\left(t x_0\right)=t^{\frac{n-1}{2}}(-\Delta)^{s/2}u(t x_0).\]
By Proposition \ref{thlaps/2},
\[\varphi(t)=c_1\{\log^-\delta(tx_0)+c_2\chi_{[0,1]}\}v(tx_0)+h_0(t)\]
in $[0,\infty)$, where $v$ is a $C^\infty(\R^n)$ extension of $u/d^s|_\Omega$ and $h_0$ is a $C^{\alpha}([0,\infty))$ function.

Arguing as in the proof of Proposition 1.6 in \cite{RS-Poh}, we find that
\[\varphi(t)=c_1\{\log^-|t-1|+c_2\chi_{[0,1]}(t)\}v(x_0)+h(t),\]
for some $h\in C^{\alpha/2}([0,\infty))$.
Note that we write $\varphi$ in this form in order to apply Proposition \ref{propoperador}.
Let us next check the hypotheses of this proposition.

Since
\[h(t)=t^{\frac{n-1}{2}}(-\Delta)^{s/2} u \left(t x_0\right)-c_1\{\log^-|t-1|+c_2\chi_{[0,1]}(t)\}v(x_0),\]
then it follows from \eqref{numeret} in Proposition \ref{thlaps/2} that $h$ satisfies condition (ii) of Proposition \ref{propoperador}.

Let now $s_0$ be such that $s/2=k+s_0$, with $k$ integer.
Then, for $x\in \R^n\backslash(2\Omega)$, one has
\[|\partial_{i}(-\Delta)^{s/2}u(x)|\leq C|x|^{-n-s_0-1}\ \ \textrm{ and }\ \ |\partial_{ij}(-\Delta)^{s/2}u(x)|\leq C|x|^{-n-s_0-2},\]
since $(-\Delta)^k u(x)$ has support in $\overline\Omega$.
This yields $|\varphi'(t)|\leq Ct^{\frac{n-1}{2}-n-s_0-1}\leq Ct^{-2-s_0}$ and $|\varphi''(t)|\leq Ct^{\frac{n-1}{2}-n-s_0-2}\leq Ct^{-3-s_0}$ for $t>2$.
This is condition (iii) of Proposition \ref{propoperador}.

Thus, it follows from Proposition \ref{propoperador} that
\[\left.\frac{d}{d\lambda}\right|_{\lambda=1^+}\int_0^\infty \varphi(\lambda t)\varphi\left(\frac{t}{\lambda}\right)dt=\left(v(x_0)\right)^2c_1^2(\pi^2+c_2^2),\]
and therefore
\[\left.\frac{d}{d\lambda}\right|_{\lambda=1^+}\int_0^\infty t^{n-1}w(\lambda tx_0)w\left(\frac{tx_0}{\lambda}\right)dt=\left(v(x_0)\right)^2c_1^2(\pi^2+c_2^2)\]
for each $x_0\in\partial\Omega$.

Furthermore, by uniform convergence on $x_0$ of the limit defining this derivative, this leads to
\[\left.\frac{d}{d\lambda}\right|_{\lambda=1^+}I_\lambda=c_1^2(\pi^2+c_2^2)\int_{\partial\Omega}(x_0\cdot\nu)
\left(\frac{u}{d^s}(x_0)\right)^2d\sigma(x_0).\]
Here we have used that $v(x)=u/d^s$ in $\overline\Omega$.
This last equality yields
\[\int_\Omega(\nabla u\cdot x)(-\Delta)^su\, dx=\frac{2s-n}{2}\int_{\Omega}u(-\Delta)^s u\, dx+c_s\int_{\partial\Omega}(x\cdot\nu)
\left(\frac{u}{d^s}(x)\right)^2d\sigma(x),\]
where $c_s=c_1^2(\pi^2+c_2^2)$ is a constant that depend only on $s$.

Hence, to finish the proof it only remains to show that
\[c_s=\Gamma(1+s)^2.\]
But since $c_s$ depends only on $s$, the constant $c_s$ in the identity can be computed by plugging an explicit solution in it.
This was done in \cite[Remark A.4]{RS-Poh} for the explicit solution given by Lemma \ref{computation} in case $s\in(0,1)$.
In case $s>1$ the computation is completely analogous, and thus it follows that $c_s=\Gamma(1+s)^2$, as desired.
\end{proof}

\section{Case II: non-star-shaped domains}
\label{sec4}

In this section we prove Theorem \ref{intparts} for general smooth domains.
We follow Section 8 in \cite{RS-Poh}.
More precisely, we will show that the bilinear identity
\begin{equation}\label{bilinear}
\begin{split}
\int_\Omega(x\cdot\nabla &u_1)(-\Delta)^su_2\, dx+\int_\Omega(x\cdot\nabla u_2)(-\Delta)^su_1\,dx=\frac{2s-n}{2}\int_\Omega u_1(-\Delta)^su_2\, dx+\\
&+\frac{2s-n}{2}\int_{\Omega}u_2(-\Delta)^su_1\, dx-\Gamma(1+s)^2\int_{\partial\Omega}\frac{u_1}{d^{s}}\frac{u_2}{d^{s}}(x\cdot\nu)\, d\sigma.
\end{split}\end{equation}
holds for all functions $u_1$ and $u_2$ satisfying the hypotheses of Theorem \ref{intparts}.

We start with a lemma, which states that the identity \eqref{bilinear} holds whenever the two functions have disjoint support.

\begin{lem}\label{duesboles}
Let $s>0$ be any noninteger, and let $u_1$ and $u_2$ be $C^s(\R^n)$ functions with disjoint compact supports $K_1$ and $K_2$.
Then,
\[\begin{split}\int_{K_1}(x\cdot\nabla u_1)(-\Delta)^su_2\, dx&+\int_{K_2}(x\cdot\nabla u_2)(-\Delta)^su_1\,dx=\\
&=\frac{2s-n}{2}\int_{K_1}u_1(-\Delta)^su_2\, dx+\frac{2s-n}{2}\int_{K_2}u_2(-\Delta)^su_1\, dx.\end{split}\]
\end{lem}

\begin{proof}
Let $s=s_0+k$, with $k\in\mathbb Z$ and $s_0\in(0,1)$.
First, it is immediate to check by induction on $k$ that, for all $w\in C^\infty(\R^n)$,
\begin{equation}\label{aa}
(-\Delta)^k(x\cdot\nabla w)=x\cdot\nabla(-\Delta)^kw+2k(-\Delta)^kw\quad \textrm{in}\ \R^n.\end{equation}
Moreover, it was proved in \cite[Lemma 5.1]{RS-Poh} that
\[(-\Delta)^{s_0}(x\cdot\nabla w)=x\cdot\nabla(-\Delta)^{s_0}w+2s(-\Delta)^{s_0}w\qquad \mbox{in}\ \  \R^n\backslash \textrm{supp}\,w\]
whenever these integrals are well defined.
Hence, we find
\begin{equation}\label{klm4}\begin{split}
(-\Delta)^s(x\cdot\nabla u_i)&=(-\Delta)^k\left\{x\cdot\nabla(-\Delta)^{s_0}u_i+2s_0(-\Delta)^{s_0}u_i\right\}\\
&=x\cdot\nabla(-\Delta)^su_i+2s(-\Delta)^su_i\qquad \mbox{in}\ \  \R^n\backslash K_i.
\end{split}\end{equation}
Note that $(-\Delta)^{s_0}u_i$ is smooth outside $K_i$, and thus we can use \eqref{aa}.

Moreover, note that, for all functions $w_1$ and $w_2$ in $L^1(\R^n)$ with disjoint compact supports $W_1$ and $W_2$, it holds the integration by parts formula
\begin{equation}\label{above}
\begin{split}
\int_{W_1}w_1(-\Delta)^sw_2&=-c_{n,s_0}\int_{W_1}\int_{W_2}w_1(x)w_2(y)(-\Delta)^k|x-y|^{-n-2s_0}dy\,dx\\
&=\int_{W_2}w_2(-\Delta)^sw_1.\end{split}
\end{equation}
Then, using \eqref{klm4} and \eqref{above}, the proof finishes as in \cite[Lemma 5.1]{RS-Poh}.
\end{proof}

The second lemma states that the bilinear identity \eqref{bilinear} holds whenever the two functions $u_1$ and $u_2$ have compact supports in a ball $B$ such that $\Omega\cap B$ is star-shaped with respect to some point $z_0$ in $\Omega\cap B$.

\begin{lem}\label{unabola} Let $\Omega$ be a bounded smooth domain, and let $B$ be a ball in $\R^n$.
Assume that there exists $z_0\in \Omega\cap B$ such that
\[(x-z_0)\cdot\nu(x)>0\qquad\mbox{for all}\ x\in\partial\Omega\cap \overline B.\]
Let $u$ be a function satisfying the hypothesis of Theorem \ref{intparts}.
Assume in addition that $(-\Delta)^su\in C^\infty(\overline\Omega)$.
Let $u_1=u\eta_1$ and $u_2=u\eta_2$, where $\eta_i\in C^\infty_c(B)$, $i=1,2$.
Then, the following identity holds
\[\int_B(x\cdot\nabla u_1)(-\Delta)^su_2\, dx+\int_B(x\cdot\nabla u_2)(-\Delta)^su_1\,dx=\frac{2s-n}{2}\int_Bu_1(-\Delta)^su_2\, dx+\]
\[+\frac{2s-n}{2}\int_{B}u_2(-\Delta)^su_1\, dx
-\Gamma(1+s)^2\int_{\partial\Omega\cap B}\frac{u_1}{d^{s}}\frac{u_2}{d^{s}}(x\cdot\nu)\, d\sigma.\]
\end{lem}

\begin{proof}
Let $\eta\in C^\infty_c(B)$ and $\tilde u=u\eta$.
We next show that
\begin{equation}\label{y}
\int_B(x\cdot \nabla \tilde u)(-\Delta)^{s}\tilde u\,dx=\frac{2s-n}{2}\int_B\tilde u(-\Delta)^s\tilde u\,dx-\Gamma(1+s)^2\int_{\partial\Omega\cap B}\left(\frac{\tilde u}{d^{s}}\right)^2(x\cdot\nu) d\sigma.
\end{equation}
From this, the lemma follows by applying \eqref{y} with $\tilde u$ replaced by $(\eta_1+\eta_2)u$ and by $(\eta_1-\eta_2)u$, and subtracting both identities.

We next prove \eqref{y}.
For it, we will apply the result for strictly star-shaped domains, already proven in Section \ref{sec2}.
As in \cite{RS-Poh}, there exists a smooth domain $\tilde \Omega$ satisfying
\[ \{\tilde u>0\}\subset \tilde\Omega\subset \Omega\cap B\quad \mbox{and}\quad(x-z_0)\cdot\nu(x)>0 \quad \mbox{for all }x\in \partial \tilde \Omega.\]

\begin{figure}
\begin{center}
\includegraphics[]{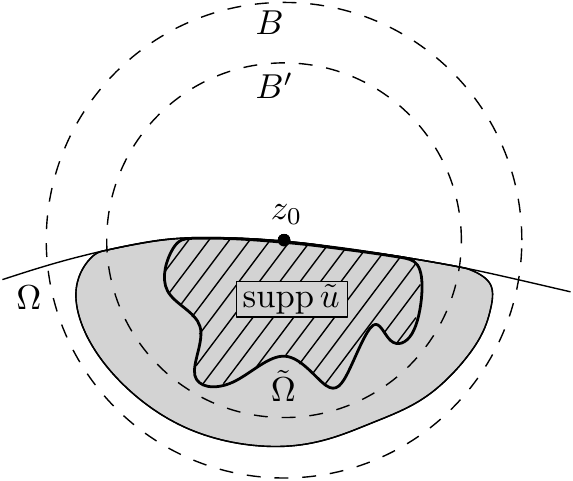}
\end{center}
\caption{\label{figura2} }
\end{figure}

Hence, it suffices to show that $\tilde u$ satisfies the hypotheses of Theorem \ref{intparts} in $\tilde\Omega$.
Let $\tilde d$ be as in Definition \ref{d} with $\Omega$ replaced by $\tilde \Omega$.
It is clear that $\tilde u\in H^s(\R^n)$ and that $\tilde u\equiv0$ in $\R^n\setminus\tilde\Omega$.
Thus, it remains only to prove that $(-\Delta)^s\tilde u\in C^\infty(\overline\Omega)$ in order to apply the results of Section \ref{sec2}.
By Theorem \ref{thGrubb}, this is equivalent to prove that $\tilde u/\tilde d\in C^\infty(\overline{\tilde\Omega})$.

But since $\textrm{supp}\,\tilde u\cap \partial\tilde\Omega\subset\partial\Omega$, then $d/\tilde d$ is smooth in $\textrm{supp}\,\tilde u$.
Thus,
\[\tilde u/\tilde d^s=\frac{u}{d^s}\frac{d^s}{\tilde d^s}\eta\in C^\infty(\overline{\tilde\Omega}),\]
where we have used that $\eta$ and $d/\tilde d$ are smooth in $\textrm{supp}\,\tilde u$, and $u/d^s\in C^\infty(\overline\Omega)$.
\end{proof}

We now give the

\begin{proof}[Proof of Theorem \ref{intparts}]
When $(-\Delta)^su\in C^\infty(\overline \Omega)$, the result follows from the two previous lemmas, by exactly the same argument as in \cite{RS-Poh}.

The result for $u\in H^s(\R^n)$ satisfying $(-\Delta)^su\in L^\infty(\Omega)$ is obtained by a density argument, as follows.
Let $g=(-\Delta)^su$, and construct a sequence of smooth functions $g_k\in C^\infty(\overline\Omega)$ satisfying $g_k\rightarrow g$ in $L^1(\Omega)$.
Let $u_k\in C^s(\R^n)$ be the solution of $(-\Delta)^s u_k=g_k$ in $\Omega$, $u_k\equiv0$ outside, and apply Theorem \ref{intparts} to these functions.
We find
\begin{equation}\label{yyy}
\int_\Omega(x\cdot \nabla u_k)(-\Delta)^{s}u_k\,dx=\frac{2s-n}{2}\int_\Omega u_k(-\Delta)^s u_k\,dx-\Gamma(1+s)^2\int_{\partial\Omega}\left(\frac{ u_k}{d^{s}}\right)^2(x\cdot\nu) d\sigma.
\end{equation}
By the estimates in Theorem \ref{thGrubb} one obtains that $u_k$, $\nabla u_k$, and $u_k/d^s$ converge uniformly to $u$, $\nabla u$, and $u/d^s$, respectively.
Thus, taking $k\rightarrow\infty$ in \eqref{yyy}, we find
\[\int_\Omega(x\cdot \nabla u)(-\Delta)^{s}u\,dx=\frac{2s-n}{2}\int_\Omega u(-\Delta)^s u\,dx-\Gamma(1+s)^2\int_{\partial\Omega}\left(\frac{ u}{d^{s}}\right)^2(x\cdot\nu) d\sigma,\]
as desired.
\end{proof}

Finally, we give the

\begin{proof}[Proof of Theorem \ref{corintparts} and Corollaries \ref{thpoh} and \ref{cor-uniquecont}]
We skip the proofs of Theorem \ref{corintparts} and Corollary \ref{thpoh} because the arguments are exactly the same as in \cite{RS-Poh}.
Let us prove next Corollary \ref{cor-uniquecont}.

First, note that any solution $\phi\in H^s(\R^n)$ to
\[\left\{ \begin{array}{rcll}
(-\Delta)^s \phi &=&\lambda\phi&\textrm{in }\Omega \\
\phi&=&0&\textrm{in }\mathbb R^n\backslash \Omega.\end{array}\right.\]
belongs to $L^\infty(\Omega)$.
Indeed, by using a standard bootstrap argument, this is an easy consequence of the results of Grubb \cite{Grubb}; see \cite{Grubb-3} where this argument is done in detail.

Hence, once we know that $\phi$ is bounded, we can apply Corollary \ref{thpoh}, to find
\[2s\int_\Omega \phi^2dx=\Gamma(1+s)^2\int_{\partial\Omega}\left(\frac{\phi}{d^s}\right)^2(x\cdot\nu)d\sigma=0.\]
It follows that, $\int_\Omega \phi^2dx=0$, and therefore $\phi\equiv0$.
\end{proof}

\appendix
\section{Proof of Lemma \ref{computation}}
\label{sec5}

In this section we prove Lemma \ref{computation}, used in Section \ref{sec2} to compute the constant $\Gamma(1+s)^2$ in the Pohozaev identity.

\begin{proof}[Proof of Lemma \ref{computation}]
Let $s=k+s_0$, with $k$ integer and $s_0\in(0,1)$.
By the explicit computations of Dyda \cite{Dyda}, we have
\begin{equation}\label{dfg}
(-\Delta)^{s_0}(1-|x|^2)_+^s=\frac{2^{2s_0}\Gamma\left(\frac{n}{2}+s_0\right)\Gamma\left(k+1+s_0\right)}{\pi^{n/2}\Gamma\left(\frac{n}{2}\right)\Gamma\left(k+1\right)}
\,_2F_1\left(\frac{n}{2}+s_0,-k;\frac{n}{2};|x|^2\right).\end{equation}
Here, $\,_2F_1$ is the hypergeometric function, defined by
\[\,_2F_1(a,b;c;z)=\sum_{m\geq0} \frac{(a)_m(b)_m}{(c)_m}\frac{z^m}{m!},\]
and $(q)_m$ is given by $(q)_m=q(q+1)\cdots(q+m-1)$.

Note that, since $-k$ is a negative integer, $(-k)_m=0$ for all $m>k$, and thus the function \eqref{dfg} is a polynomial of degree $2k$.
Moreover, its leading coefficient is
\[\,_2F_1\left(\frac{n}{2}+s_0,-k;\frac{n}{2};|x|^2\right)=\frac{\bigl(\frac{n}{2}+s_0\bigr)_k (-k)_k}{\bigl(\frac{n}{2}\bigr)_k}\frac{|x|^{2k}}{k!}+...\]
Thus,
\[\begin{split}
(-\Delta)^s(1-|x|^2)_+^s&=(-\Delta)^k(-\Delta)^{s_0}(1-|x|^2)_+^s\\
&=
\frac{2^{2s_0}\Gamma\left(\frac{n}{2}+s_0\right)\Gamma\left(k+1+s_0\right)}{\pi^{n/2}\Gamma\left(\frac{n}{2}\right)\Gamma\left(k+1\right)}
\frac{\bigl(\frac{n}{2}+s_0\bigr)_k (-k)_k}{\bigl(\frac{n}{2}\bigr)_k}
\frac{(-\Delta)^k|x|^{2k}}{k!}.\end{split}\]
Now, it is immediate to show that
\[(-\Delta)^k|x|^\beta=\beta(\beta-2)\cdots(\beta-2k+2)\times (2-n-\beta)(4-n-\beta)\cdots(2k-n-\beta)\times|x|^{\beta-2k}.\]
In particular,
\[(-\Delta)^k|x|^{2k}=2^kk!\times(-1)^k 2^k\left(\frac{n}{2}\right)_k.\]
Therefore,
\[(-\Delta)^s(1-|x|^2)_+^s=
\frac{2^{2s_0}\Gamma\left(\frac{n}{2}+s_0\right)\Gamma\left(k+1+s_0\right)}{\pi^{n/2}\Gamma\left(\frac{n}{2}\right)\Gamma\left(k+1\right)}
\frac{\bigl(\frac{n}{2}+s_0\bigr)_k (-k)_k}{\bigl(\frac{n}{2}\bigr)_k}
\frac{2^kk!\times(-1)^k 2^k\left(\frac{n}{2}\right)_k}{k!}.\]
Now, since $(-k)_k=(-1)^kk!$, and $\Gamma(1+z)=z\Gamma(z)$, we find
\[(-\Delta)^s(1-|x|^2)_+^s=
\frac{2^{2s}\Gamma\left(\frac{n}{2}+s\right)\Gamma\left(1+s\right)}{\pi^{n/2}\Gamma\left(\frac{n}{2}\right)},\]
as desired.
\end{proof}

\section*{Acknowledgements}

We thank G. Grubb for her comments and remarks on a previous version of this paper.

\end{document}